\documentclass[12pt]{article}

\usepackage{amsmath,amsfonts,amssymb}
\usepackage{bbm}

\usepackage{lipsum}

\usepackage{xcolor}
\usepackage{ulem}

\newtheorem{defn}{Definition}[section]
\newtheorem{theo}[defn]{Theorem}

\newtheorem{prop}[defn]{Proposition}
\newtheorem{cor}[defn]{Corollary}
\newtheorem{rem}[defn]{\rm Remark}
\newtheorem{exam}[defn]{\rm Example}

\newenvironment{proof}{{\rm Proof. }}{{\vskip 0.1cm \hfill$\Box$}}

\def\N {{\mathbb N}}

\def\R {{\mathbb R}}
\def\E{{\mathbb E}}
\def\P{{\mathbb P}}
\def\M{{\mathbb M}}
\newcommand{\F}{\mathcal{F}}

\begin{document}


\noindent
{{\Large\bf On the pathwise uniqueness for a class of degenerate It\^{o}-stochastic differential equations}
{\footnote[1]{The research of Haesung Lee was supported by Basic Science Research Program through the National Research Foundation of Korea (NRF) funded by the Ministry of Education (2020R1A6A3A01096151).  \\[4pt]
Mathematics Subject Classification (2020): Primary 60H10, 47D07, Secondary 35B65, 60J60}} \\ \\
\bigskip
\noindent
{\bf Haesung Lee}  \\
\noindent
{\small{\bf Abstract.}  
We show pathwise uniqueness for a class of degenerate It\^{o}-SDE among all of its weak solutions that spend zero time at the points of degeneracy of the dispersion matrix. Consequently, by the Yamada-Watanabe Theorem and a weak existence result, the pathwise unique solutions can be shown to be strong and to exist. The main tools to show pathwise uniqueness are inequalities associated with maximal functions and a Krylov type estimate derived from elliptic regularity and uniqueness in law. \\ \\
\noindent 
{Keywords: pathwise uniqueness, degenerate It\^{o}-SDE, Krylov type estimate, uniqueness in law, maximal function

\section{Introduction}
An It\^{o}-SDE whose diffusion coefficient is not locally uniformly elliptic is called a degenerate It\^{o}-SDE and in this paper, we aim to show pathwise uniqueness for a class of such SDEs on $\R^d$ with $d \geq 3$. In \cite{X11}, a Krylov type estimate induced by a parabolic regularity result on $\R^d$ is essentially used to show pathwise uniqueness for a class of non-degenerate It\^{o}-SDEs with singular drift coefficients. However, it is not clear how such an estimate can be derived when the diffusion coefficient is not locally uniformly elliptic. Recently in \cite{LT19de}, using elliptic and parabolic regularity results and generalized Dirichlet form theory, the existence and uniqueness in law of a quite large class of degenerate It\^{o}-SDEs with fully discontinuous coefficients was shown (see Theorem \ref{theo:2.1}). This leads us to explore a subclass of degenerate It\^{o}-SDE for which a Krylov type estimate can be concluded from resolvent regularity. \\
Let us introduce our main results. For some $y \in \R^d$, we deal with the following time-homogeneous degenerate It\^{o}-SDE 
\begin{equation} \label{sdes}
X_t = y+\int_0^t \sqrt{\frac{1}{\psi}}\,(X_s) \cdot \sigma (X_s) dW_s + \int_0^t \mathbf{G}(X_s) ds, \quad 0 \leq t<\infty, \;
\end{equation}
where $(W_t)_{t \geq 0}$ is a $d$-dimensional standard Brownian motion and the degeneracy stems from $\sqrt{\frac{1}{\psi}}$. Now consider
\\
\text{}\\
{\bf (H1)}: {\it $d \geq 3$, $\mathbf{G}=(g_1, \ldots, g_d) \in L^{\infty}_{loc}(\R^d, \R^d)$ and $A=(a_{ij})_{1 \leq i,j \leq d}$ is a symmetric matrix of functions with $a_{ji}=a_{ij} \in H_{loc}^{1,2d+2}(\mathbb{R}^d) \cap C(\mathbb{R}^d)$ for all $1 \leq i,j \leq d$ such that for every open ball $B\subset \R^d$, there exist constants $\lambda_B, \Lambda_B>0$ with
\begin{equation*}
\lambda_{B} \|\xi\|^2 \leq \langle A(x) \xi, \xi \rangle \leq \Lambda_B \|\xi\|^2, \quad \text{ for all } \xi \in \mathbb{R}^d, \; x \in B.
\end{equation*}
$\sigma=(\sigma_{ij})_{1 \leq i,j \leq d}$ is a matrix of continuous functions with $\sigma \sigma^T = A$. $\psi \in L_{loc}^{q}(\mathbb{R}^d)$ for some $q>2d+2$ with $\psi>0$ a.e., $\frac{1}{\psi} \in L_{loc}^{\infty}(\mathbb{R}^d)$, and $\sqrt{\frac{1}{\psi}} (x)\in [0,\infty)$ for any $x\in \R^d$.  
}\\
\text{}\\[5pt]
If {\bf (H1)} holds and $\psi$, $\sigma$ and $\mathbf{G}$ are as in {\bf (H1)}, then we will consider the following conditions.\\ \\
{\bf (H2)}:
{\it
There exist constants $N_0 \in \N$ and $M>0$ such that for a.e.  $x \in \R^d \setminus \overline{B}_{N_0}$
\begin{eqnarray}\label{conscondit}
-\frac{\langle \frac{1}{\psi}(x)A(x)x, x \rangle}{\left \| x \right \|^2}+ \frac{\mathrm{trace}(\frac{1}{\psi}(x)A(x))}{2}+ \big \langle \mathbf{G}(x), x \big \rangle  \leq M \|x\|^2 (\ln\|x\|  +1). 
\end{eqnarray}}
}
\text{}\\
{
\it
{\bf (H3)}: There exists an open set $E$ in $\R^d$ such that for any open ball $U$ with $\overline{U} \subset E$, 
$$
0<\inf_{U} \psi \leq \sup_{U} \psi<\infty.
$$
}
\\
{
\it
{\bf (H4)}: For some $\widetilde{q} \in (\frac{d}{2}, \infty)$, $\sqrt{\frac{1}{\psi}} \cdot \sigma_{ij} \in H^{1, 2 \widetilde{q}}_{loc}(\R^d)$ and $g_i \in H^{1, \widetilde{q}}_{loc}(\R^d)$ for all $1 \leq i, j \leq d$.
}
\\
\\
Then our main result is stated as follows.
\begin{theo}\label{theo:1.1}
Assume {\bf (H1)}--{\bf (H4)} and let $E$ be as in {\bf (H3)}. Then for $y \in E$, pathwise uniqueness holds  for \eqref{sdes} in the following sense. If $(\widetilde{\Omega}, \widetilde{\F}, \widetilde{\P}_{y}, (\widetilde{\F}_t)_{t \ge 0}, (\widetilde{X}^k_t)_{t \ge 0}, (\widetilde{W}_t)_{t \ge 0})$, $k \in \{1,2 \}$
 are weak solutions to \eqref{sdes} satisfying both
\begin{equation} \label{zerotime}
\int_0^{\infty} 1_{\{\sqrt{\frac{1}{\psi}}=0\}}(\widetilde{X}^k_s) ds =0, \quad \text{ $\widetilde{\P}_{y}$-a.s},
\end{equation}
then
\begin{equation} \label{pathuni}
\widetilde{\P}_{y}(\widetilde{X}_t^1 =\widetilde{X}_t^2, \; t \geq 0 )=1.
\end{equation}
\medskip
Moreover, for $y \in E$ and a $d$-dimensional Brownian motion $(\widetilde{W}_t)_{t \geq 0}$ on a probability space $(\widetilde{\Omega}, \widetilde{\mathcal{F}}, \widetilde{\P})$,
there exists a strong solution $(Y^{y}_t)_{t \geq 0}$ to \eqref{sdes} satisfying
\begin{equation*}
\int_0^{\infty} 1_{\{\sqrt{\frac{1}{\psi}}=0\}}(Y^y_s) ds =0, \quad \text{ $\widetilde{\P}$-a.s.}
\end{equation*}
\end{theo}
The main ingredient for the proof of Theorem \ref{theo:1.1} is a new Krylov type estimate (Corollary \ref{kryloves}) which is derived from uniqueness in law (Theorem \ref{theo:2.1}) and elliptic regularity results for the resolvent (see \eqref{rescontest}). Additionally, we adapt a technique used in \cite[Proof of Theorem 2.2]{X11} and  well-known inequalities for the Hardy–Littlewood maximal function (Proposition \ref{maxth}). As a concrete application of Theorem \ref{theo:1.1}, we obtain the following result.

\begin{cor} \label{cor:1.2}
Let $d \geq 3$ and $\alpha \in [0, \frac{d}{2d+2})$ be a constant, $\mathbf{G}=(g_1, \ldots, g_d) \in L^{\infty}_{loc}(\R^d, \R^d)$ satisfying that for some constants $N_0 \in \N$ and $M>0$,
$$
\langle \mathbf{G}(x), x \big \rangle  \leq M \|x\|^2 (\ln\|x\|  +1), \qquad \text{ for a.e. } x \in \R^d \setminus \overline{B}_{N_0}
$$
and that for some $\varepsilon>0$ $g_i \in H^{1, \frac{d}{2}+\varepsilon}_{loc}(\R^d)$ for all $1 \leq i \leq d$. Then for $y \in \R^d \setminus \{0\}$, pathwise uniqueness holds for 
\begin{equation} \label{sdesspe}
X_t = y+\int_0^t \|X_s\|^{\frac{\alpha}{2}} \cdot id \; dW_s + \int_0^t \mathbf{G}(X_s) ds, \quad 0 \leq t<\infty \;
\end{equation}
in the following sense. If $(\widetilde{\Omega}, \widetilde{\F}, \widetilde{\P}_{y}, (\widetilde{\F}_t)_{t \ge 0}, (\widetilde{X}^k_t)_{t \ge 0}, (\widetilde{W}_t)_{t \ge 0})$, $k \in \{1,2 \}$ are weak solutions to \eqref{sdesspe} satisfying both
$$
\int_0^{\infty} 1_{\{ 0\}}(\widetilde{X}^k_s) ds =0, \quad \text{ $\widetilde{\P}_{y}$-a.s},
$$
then \eqref{pathuni} follows. Moreover, for $y \in \R^d \setminus \{0\}$ and a $d$-dimensional  Brownian motion $(\widetilde{W}_t)_{t \geq 0}$ on a probability space $(\widetilde{\Omega}, \widetilde{\mathcal{F}}, \widetilde{\P})$ there exists a strong solution $(Y^{y}_t)_{t \geq 0}$ to \eqref{sdesspe} satisfying
\begin{equation*}
\int_0^{\infty} 1_{\{0\}}(Y^y_s) ds =0, \quad \text{ $\widetilde{\P}$-a.s.}
\end{equation*}
\end{cor}
To the author's knowledge, Engelbert and Schmidt first showed in \cite{ES85} the existence and uniqueness in law of one-dimensional SDEs without drift among the solutions that spend zero time at the zeros of dispersion coefficients. In subsequent studies, they presented in \cite{ES85b} and \cite{ES91} a sufficient condition for the existence of a pathwise unique and strong solution for one-dimensional SDEs with general drift. As another related literature, we refer to \cite{MS73} and \cite{BC05} dealing with similar types of results above. Although here formally similar results are shown for $d \geq 3$, the used methods completely differ from  \cite{MS73}, \cite{BC05}, \cite{ES85}, \cite{ES85b}, \cite{ES91} in that we ultimately develop a new Krylov type estimate by using elliptic regularity and uniqueness in law.
\\
In Section \ref{frame}, we introduce well-known results about the Hardy–Littlewood maximal function and briefely explain the results of \cite{LT19de}, where the existence and uniqueness in law of weak solutions to \eqref{sdes} are shown.  In Section \ref{sec3}, we prove the already mentioned a new Krylov type estimate for weak solutions to \eqref{sdes} satisfying \eqref{zerotime}. Finally, in Section \ref{pathw} we present the proofs of Theorem \ref{theo:1.1} and Corollary \ref{cor:1.2} and some examples where Theorem \ref{theo:1.1} is applied.

\section{Preliminaries} \label{frame}
For basic notations which are not defined in this paper, we refer to \cite[\bf Notations and Conventions]{LT20}. 
For $r>0$ and $x \in \R^d$, denote by $B_r(x)$ the open ball with radius $r$ with center $0$ and let $B_r=B_r(0)$.
Weak and strong solutions to It\^{o}-SDEs are defined as in \cite[Definition 3.50]{LT20}.  For a matrix of functions $B=(b_{ij})_{1 \leq i,j \leq d}$, we write
$\|B\| = \left( \sum_{i,j=1}^d |b_{ij}|^2  \right)^{1/2}$. Denote the $d \times d$ identity matrix by $id$. $\mathcal{M}$ denotes the Hardy–Littlewood maximal operator defined by
$$
\mathcal{M}f(x) = \sup_{r>0} \,\frac{1}{dx(B_r(x))} \int_{B_{r}(x)} |f(y)| dy, \qquad  f \in L^1_{loc}(\R^d), \;\, x \in \R^d.
$$
The following well-known results are crucially used in the proof of Theorem \ref{theo:1.1}.
\begin{prop} \label{maxth}
\begin{itemize}
\item[(i)]
If $f \in L^r(\R^d)$ with $r \in (1, \infty]$, then $\mathcal{M}f \in L^r(\R^d)$ and there exists a constant $c_{d,r}>0$ which only depends on $d$ and $r$ such that
$$
\|\mathcal{M}f \|_{L^r(\R^d)} \leq c_{d,r} \|f\|_{L^r(\R^d)}.
$$
\item[(ii)]
There exists a constant $\widetilde{c}_d>0$ which only depends on $d$ such that for any $f \in C_0^{1}(\R^d)$ and $x, y \in \R^d$ it holds
$$
|f(x)-f(y)| \leq \widetilde{c}_d\|x-y\| (\mathcal{M}\|\nabla f\|(x)+\mathcal{M}\|\nabla f\|(y)).
$$
\end{itemize}
\end{prop}
\begin{proof}
(i) directly follows from \cite[Chapter 1, 1.3, Theorem 1 (c)]{S93}. For the proof of (ii), let us choose a measurable set $A \subset \R^d$ in \cite[Lemma 3.5]{X13} such that $dx(A)=0$ and that
\begin{equation} \label{negiest}
|f(z)-f(w)| \leq 2^d \|z-w\| \left(\mathcal{M}\|\nabla f\|(z)+\mathcal{M}\|\nabla f\|(w)\right), \quad \forall z, w \in \R^d \setminus A.
\end{equation}
Since $\| \nabla f\| \in C(\R^d)$, it holds $\|\nabla f\|(x) \leq \mathcal{M}\|\nabla f\| (x)$ for all  $x \in \R^d$. Thus, by \cite[Lemma 3.4]{AP07} we obtain $\mathcal{M} \| \nabla f \| \in C(\R^d)$, hence the assertion follows from \eqref{negiest} and the continuity of $f$.

\end{proof}
\text{}\\
Under the assumption {\bf (H1)}, by \cite[Theorem 4]{LT19de} there exists $\mu=\rho \psi dx$ with $\rho \in H^{1,2d+2}_{loc}(\R^d) \cap C(\R^d)$ and $\rho(x)>0$ for all $x \in \R^d$ such that
$$
\int_{\R^d} \langle \mathbf{G}- \frac{1}{2\psi} \nabla A- \frac{1}{2\rho \psi}A \nabla \rho, \nabla f \rangle d \mu = 0, \quad \forall f \in C_0^{\infty}(\R^d).
$$
Moreover, by \cite[Theorem 3]{LT19de} there exists a sub-Markovian $C_0$-semigroup of contractions $(\overline{T}_t)_{t>0}$ on $L^1(\R^d, \mu)$ whose generator extends $(L, C_0^{\infty}(\R^d))$, where
$$
Lf= \frac12 \text{trace}(\frac{1}{\psi}A \nabla^2 f) +\langle \mathbf{G}, \nabla f \rangle, \quad f \in C_0^{\infty}(\R^d).
$$
Through Riesz–Thorin interpolation, $(\overline{T}_t)_{t>0}$ restricted to $L^1(\R^d, \mu)_b$ can be extended to a sub-Markovian $C_0$-semigroup of contractions on each $L^r(\R^d, \mu)$, $r \in [1, \infty)$ and to a sub-Markovian semigroup of contractions on $L^{\infty}(\R^d, \mu)$. We donote all these by $(T_t)_{t>0}$. Denote by $(G_{\alpha})_{\alpha>0}$ the sub-Markovian $C_0$-resolvent of contractions on each $L^{r}(\R^d, \mu)$, $r \in [1, \infty)$ and the sub-Markovian resolvent of contractions on $L^{\infty}(\R^d, \mu)$ associated with $(T_t)_{t>0}$, i.e. 
$$
G_{\alpha} f := \int_0^{\infty} e^{-\alpha s} T_s f  ds, \; \qquad f \in \cup_{r \in [1, \infty] } L^r(\R^d, \mu).
$$
Then by \cite[Theorems 5, 6]{LT19de}, there exist $(R_{\alpha})_{\alpha>0}$ and $(P_t)_{t>0}$ such that for any $\alpha, t>0$ and $f \in \mathcal{B}_b(\R^d)$
$$
R_{\alpha} f \in C(\R^d), \;\; P_{\cdot} f \in C(\R^d \times (0, \infty)) \;\; \text{ and }\;\; R_{\alpha} f = G_{\alpha} f, \;\; \;  P_t g =T_t g \quad \text{$\mu$-a.e.}
$$
Moreover, it follows from \cite[Theorem 7]{LT19de} that there exists a Hunt process 
$$
\M =\left(\Omega, \mathcal{F}, (\mathcal{F}_t)_{t \geq 0}, (X_t)_{t \geq 0}, (\P_x)_{x \in \R^d}  \right)
$$
with state space $\R^d$ and life time $\zeta= \inf \{ t \geq 0:  X_t = \Delta \}$ satisfying that
$$
\P_x \left( \{ \omega \in \Omega : X_{\cdot}(\omega) \in C([0, \infty), \R^d), \; X_t (\omega)=\Delta \;\text{ for all } t \geq \zeta   \}   \right)=1, \quad \forall x \in \R^d,
$$
where $\Delta$ is a point at infinity and that for all $f \in \mathcal{B}_b(\R^d)$,\; $t, \alpha>0$ and  $x \in \R^d$, 
\begin{equation} \label{resrel}
P_t f(x) = \mathbb{E}_x[f(X_t)], \quad R_{\alpha} f(x) = \E_x \left[ \int_0^{\infty} e^{-\alpha s} f(X_s) ds\right], 
\end{equation}
where $\E_x$ is the expectation with respect to $\P_x$ defined on $(\Omega, \mathcal{F})$. Under the assumptions {\bf (H1)} and {\bf (H2)}, using a similar method to \cite[Lemma 3.26]{LT20}, we obtain that  $\M$ is non-explosive, i.e. $\P_x(\zeta=\infty)=1$ for all  $x \in \R^d$ (equivalently, $(T_t)_{t>0}$ is conservative, i.e. $T_t 1_{\R^d}=1$, $\mu$-a.e. for all $t>0$). 
As a direct consequence of \cite[Theorems 8, 12, Lemma 2 and Remark 2]{LT19de} and the above, we then obtain the following theorem.
\begin{theo} \label{theo:2.1}
Assume {\bf (H1)}--{\bf (H2)} and let $\M =  (\Omega, \F, (\F_t)_{t \ge 0}, (X_t)_{t \ge 0}, (\P_x)_{x \in \R^d})$ be the Hunt process defined above. Then there exists an $(\mathcal{F}_t)_{t\geq0}$-Brownian motion $(W_t)_{t \geq 0}$ such that
for any $y \in \R^d$, $(\Omega, \F, \P_{y}, (\F_t)_{t \ge 0}, (X_t)_{t \ge 0}, (W_t)_{t \ge 0})$ is a weak solution to \eqref{sdes} satisfying
$$
\int_0^{\infty} 1_{\{\sqrt{\frac{1}{\psi}}=0\}}(X_s) ds =0, \quad \text{ $\P_{y}$-a.s}.
$$
Furthermore if for some $y \in \R^d$ there exists a weak solution $(\widetilde{\Omega}, \widetilde{\F}, \widetilde{\P}_{y}, (\widetilde{\F}_t)_{t \ge 0}, (\widetilde{X}_t)_{t \ge 0}, (\widetilde{W}_t)_{t \ge 0})$ to \eqref{sdes} such that
\begin{equation} \label{zerospen}
\int_0^{\infty} 1_{\{\sqrt{\frac{1}{\psi}}=0\}}(\widetilde{X}_s) ds =0, \quad \text{ $\widetilde{\P}_{y}$-a.s.},
\end{equation}
then
$$
\P_y \circ X^{-1} = \widetilde{\P}_y \circ \widetilde{X}^{-1}, \quad \text{ on } \mathcal{B}(C([0, \infty), \R^d)).
$$
\end{theo}

\begin{rem} \;
In Theorem \ref{theo:2.1}, the assumption $d \geq 3$ of {\bf (H1)} can be replaced by $d=2$. However, we point out that $d \geq 3$ of {\bf (H1)} is needed for all other results of this paper.  Moreover, in order to obtain Theorem \ref{theo:2.1}, the condition in {\bf (H1)} that $\sigma_{ij}$ is continuous for all $1 \leq i,j \leq d$ can be replaced by the condition that $\sigma_{ij}$ is locally bounded and measurable for all $1 \leq i,j \leq d$. 
\end{rem}

\section{Krylov type estimates} \label{sec3}
\begin{theo} \label{krylovest1}
Assume {\bf (H1)} and {\bf (H3)} and let $E$ be as in {\bf (H3)}. Let  $\widetilde{q} \in (\frac{d}{2}, \infty)$, $y \in E$, $t>0$ and $B$ be an open ball in $\R^d$. Let  $\M =  (\Omega, \F, (\F_t)_{t \ge 0}, (X_t)_{t \ge 0}, (\P_x)_{x \in \R^d})$ be the Hunt process of Section \ref{frame}. Then for any $f \in L^{\widetilde{q}}(\R^d)_0$ with $\text{supp}(f dx) \subset B$, there exists a constant $C_{y,B, \widetilde{q}}>0$ which is independent of $f$ and $t>0$ such that
$$
\mathbb{E}_{y} \left[  \int_0^t f(X_s) ds \right] \leq e^{t} C_{y,B, \widetilde{q}} \|f\|_{L^{\widetilde{q}}(\R^d)}.
$$
\end{theo}
\begin{proof}
Let $(G_{\alpha})_{\alpha>0}$ be the resolvent defined on $\cup_{r \in [1, \infty]}L^r(\R^d, \mu)$ as in Section \ref{frame}. 
Since $E$ is an open subset of $\R^d$ and $y \in E$, there exist open balls $U_1$ and $U_2$ such that $y \in U_1 \subset \overline{U}_1 \subset U_2 \subset \overline{U}_2 \subset E$. First assume $f \in C_0^{\infty}(\R^d)$ with $\text{supp}(f) \subset B$.
 Then by the proof of \cite[Theorem 5]{LT19de} (cf. \cite[Theorem 3 (c)]{LT19de}), $G_{1} f \in H^{1,2}(U_2)$ and 
\begin{eqnarray*}
&&\int_{U_2} \big \langle \frac12 \rho A \nabla G_{1}f, \nabla \varphi \big \rangle dx - \int_{U_2} \langle \rho \psi \mathbf{B}, \nabla G_{1} f \rangle \varphi\, dx  +\int_{U_2} (\alpha \rho \psi G_{1} f) \,\varphi dx  \\
&&\qquad = \int_{U_2} (\rho \psi f) \,\varphi dx, \qquad \; \forall \varphi \in C_0^{\infty}(U_2), \label{feppwikmee}
\end{eqnarray*}
where $\mathbf{B}=\mathbf{G}-\frac{1}{2\psi}(\nabla A +\frac{1}{\rho}A \nabla \rho)$. Note that  $d \geq 3$ implies $1-\frac{1}{2d+2}>\frac{2}{d}>\frac{1}{\widetilde{q}}$. By \cite[Theorems 2, 5]{LT19de}, {\bf (H3)} and the $L^1(\R^d, \mu)$-contraction property of $(\alpha G_{\alpha})_{\alpha>0}$, there exists a constant $C_1>0$ which only depends on the coefficients $A$, $\mathbf{G}$, $\psi$ and $U_1$, $U_2$, $\widetilde{q}$, $d$ such that
\begin{eqnarray}
\|R_{1} f \|_{C(\overline{U}_1)} &\leq& C_1 \big( \|G_{1} f \|_{L^1(U_2)}+ \| \rho \psi  f\|_{L^{\widetilde{q}}(U_2)}    \big) \nonumber \\ 
&\leq& C_1 \big(  (\inf_{U_2} \rho \psi)^{-1} \|G_{1} f \|_{L^1(\R^d, \mu)}+ (\sup_{U_2} \rho \psi) \| f\|_{L^{\widetilde{q}}(\R^d)}    \big) \nonumber \\
&\leq& C_1 \big(  (\inf_{U_2} \rho \psi)^{-1} \|  f \|_{L^1(\R^d, \mu)}+(\sup_{U_2} \rho \psi)  \| f\|_{L^{\widetilde{q}}(\R^d)}    \big) \nonumber \\
&\leq& C_1 \big(  (\inf_{U_2} \rho \psi)^{-1} \| \rho \psi \|_{L^{2d+2}(B)} \|  f \|_{L^{\frac{2d+2}{2d+1}}(B)}+(\sup_{U_2} \rho \psi)  \| f\|_{L^{\widetilde{q}}(\R^d)}    \big) \nonumber \\
&\leq& C_1  C_2 \|f\|_{L^{\widetilde{q}}(\R^d)}, \label{rescontest}
\end{eqnarray}
where $C_2:=  (\inf_{U_2} \rho \psi)^{-1} \| \rho \psi \|_{L^{2d+2}(B)} dx(B)^{\frac{2d+1}{2d+2}-\frac{1}{\widetilde{q}}} +  \sup_{U_2} \rho \psi$. Using the denseness of $C_0^{\infty}(B)$ in $L^{\widetilde{q}}(B)$, \eqref{rescontest} extends to $f \in \mathcal{B}_b(\R^d)_0$ with $\text{supp}(fdx) \subset B$. By \eqref{resrel} and \eqref{rescontest}, it holds for any $f \in \mathcal{B}_b(\R^d)_0$ with $\text{supp}(fdx) \subset B$
\begin{eqnarray*}
\E_y \left[ \int_0^t f(X_s) ds   \right]  &\leq& e^{t} \E_y \left[ \int_0^{\infty} e^{- s}  |f|(X_s) ds  \right] \\
&=& e^t R_1 |f| (y) \leq e^{t} \|R_{1} |f| \|_{C(\overline{U}_1)}  \leq e^{t} C_1  C_2 \|f\|_{L^{\widetilde{q}}(\R^d)}.
\end{eqnarray*}
Finally, using monotone approximation, the assertion follows.
\end{proof}
\text{} \\
As a direct consequence of Theorems \ref{krylovest1}, \ref{theo:2.1}, we obtain the following.
\begin{cor} \label{kryloves}
Assume {\bf (H1)}--{\bf (H3)} and let $E$ be as in {\bf (H3)}. Let $\widetilde{q} \in (\frac{d}{2}, \infty)$, $y \in E$, $t>0$ and $B$ be an open ball in $\R^d$. Let $(\widetilde{\Omega}, \widetilde{\F}, \widetilde{\P}_{y}, (\widetilde{\F}_t)_{t \ge 0}, (\widetilde{X}_t)_{t \ge 0}, (\widetilde{W}_t)_{t \ge 0})$ be a weak solution to \eqref{sdes} satisfying \eqref{zerospen}. Then for any $f \in L^{\widetilde{q}}(\R^d)_0$ with $\text{supp}(f dx) \subset B$, there exists a constant $C_{y,B, \widetilde{q}}>0$ which is independent of $f$ and $t>0$ such that
$$
\widetilde{\mathbb{E}}_{y} \left[  \int_0^t f(\widetilde{X}_s) ds \right] \leq e^{t} C_{y,B, \widetilde{q}} \|f\|_{L^{\widetilde{q}}(\R^d)},
$$
where $\widetilde{\mathbb{E}}_y$ is the expectation with respect to $\widetilde{\mathbb{P}}_y$ defined on $(\widetilde{\Omega}, \widetilde{\F})$.
\end{cor}

\section{Pathwise uniqueness} \label{pathw}
{\rm Proof of Theorem \ref{theo:1.1}.}\;\;
Fix $y \in E$. For $k \in \{1,2 \}$, let $(\widetilde{\Omega}, \widetilde{\F}, \widetilde{\P}_{y}, (\widetilde{\F}_t)_{t \ge 0}, (\widetilde{X}^k_t)_{t \ge 0}, (\widetilde{W}_t)_{t \ge 0})$ be weak solutions to \eqref{sdes} satisfying \eqref{zerotime} and let $\widetilde{Z}_t:=\widetilde{X}_t^{1}-\widetilde{X}_t^{2}$, $t \geq 0$. Given $n \in \N$ with $n \geq 2$ and $y \in B_{n-1}$, set $\widetilde{D}^k_n:=\inf \{ t \geq 0 \,:  \widetilde{X}_t^k \in \R^d \setminus B_{n-1} \}$ and $\widetilde{D}_n:= \widetilde{D}_n^{1} \wedge \widetilde{D}_n^{2}$. Let $\widehat{\sigma}=(\widehat{\sigma}_{ij})_{1 \leq i,j \leq d}$ be defined by $\widehat{\sigma} = \sqrt{\frac{1}{\psi}} \cdot \sigma$. By It\^{o}'s formula, for any $t\geq 0$ it holds  $\P_y$-a.s.
\begin{eqnarray*}
\|\widetilde{Z}_{t \wedge D_n}\|^2&=& 2\int_0^{t \wedge D_n} \widetilde{Z}_s \left(\widehat{\sigma}(\widetilde{X}^{1}_s)-\widehat{\sigma}(\widetilde{X}^{2}_s \right) d\widetilde{W}_s + \int_0^{t\wedge D_n} \| \widehat{\sigma}(\widetilde{X}_s^{1})- \widehat{\sigma}(\widetilde{X}_s^{2}) \|^2 ds\\
&&\qquad   +2\int_0^{t \wedge D_n} \langle \widetilde{Z}_s, \mathbf{G}(\widetilde{X}_s^{1})- \mathbf{G}(\widetilde{X}_s^{2}) \rangle ds.
\end{eqnarray*}
Let $\widehat{\sigma}^n_{ij} \in H^{1,2\widetilde{q}}(\R^d)_0$ be an extension of $1_{B_{n}}\widehat{\sigma}_{ij} \in H^{1,2\widetilde{q}}(B_{n})$ satisfying $\text{supp}(\widehat{\sigma}_{ij}^n dx) \subset B_{n+1}$  for each $i,j \in \{1,2 \ldots d\}$. Likewise, let $ g^n_i \in H^{1,\widetilde{q}}(\R^d)_0$ be an extension of $1_{B_n}g_i$ satisfying $\text{supp}(g_i^n) \subset B_{n+1}$ for each $i \in \{1,2 \ldots d\}$ and let $\mathbf{G}^n=(g^n_1, \dots, g^n_d)$. Given $m \in \N$, let $\eta_{m} \in C_0^{\infty}(B_{1/m})$ be a standard mollifier on $\R^d$. Define
\begin{eqnarray*}
\widehat{\sigma}^{n,m}_{ij}:= \widehat{\sigma}^n_{ij}* \eta_m, \;\; g^{n,m}_i:= g^n_i* \eta_m, \quad \text{ for } i,j \in \{1,2, \ldots, d\},
\end{eqnarray*}
where $f*g$ is the convolution of $f$ and $g$. Let $\widehat{\sigma}^{n,m}=(\widehat{\sigma}^{n,m}_{ij})_{1 \leq i,j \leq d}$ and $\mathbf{G}^{n,m}=(g^{n,m}_1, \ldots, g^{n,m}_d)$.
Set 
$$
N_t:=M_t+A^{(1)}_t+A^{(2)}_t \; \text{ and }\; N^n_t:=N_{t \wedge \widetilde{D}_n} \; \text{ for } t \geq 0, 
$$
where
\begin{eqnarray*}
&&M_t:= 2 \int_0^t  \frac{ \widetilde{Z}_s \left(\widehat{\sigma}(\widetilde{X}^{1}_s)-\widehat{\sigma}(\widetilde{X}^{2}_s ) \right)}{\|\widetilde{Z}_s\|^2} d\widetilde{W}_s, \;\;\; \;\; A^{(1)}_t:= \int_0^{t} \frac{\| \widehat{\sigma}(X_s^{1})- \widehat{\sigma}(\widetilde{X}_s^{2}) \|^2}{\|\widetilde{Z}_s\|^2} ds, \;\;\; \\
&&\qquad \qquad A^{(2)}_t:=  2\int_0^{t}  \frac{\langle \widetilde{Z}_s, \mathbf{G}(\widetilde{X}_s^{1})- \mathbf{G}(\widetilde{X}_s^{2}) \rangle}{\|\widetilde{Z}_s\|^2} ds, \qquad t \geq 0, 
\end{eqnarray*}
where $\frac{0}{0}:=0$. Then for all $t\geq0$ it holds 
\begin{equation} \label{undereq}
\|\widetilde{Z}_{t \wedge D_n}\|^2   =  \int_0^{t} \|\widetilde{Z}_{s \wedge D_n} \|^2 dN^n_{s}, \quad \; \text{ $\widetilde{\P}_y$-a.s. }
\end{equation}
Let $\widetilde{\E}_y$ be the expectation with respect to $\widetilde{\P}_y$ defined on $(\widetilde{\Omega}, \widetilde{\F})$.
By It\^{o}-isometry and Fatou's Lemma, it holds
\begin{eqnarray*}
\frac14\widetilde{\E}_y \left[ |M_{t\wedge \widetilde{D}_n}|^2 \right]  &\leq& \widetilde{\E}_y \left[ \int_0^{t \wedge \widetilde{D}_n} \frac{\| \widehat{\sigma}^n(\widetilde{X}_s^{1})- \widehat{\sigma}^n(\widetilde{X}_s^{2}) \|^2}{\|\widetilde{Z}_s\|^2} ds \right] \\
&\leq& \liminf_{\delta \rightarrow 0+} \widetilde{\E}_y \left[ \int_0^{t \wedge \widetilde{D}_n}  \frac{\| \widehat{\sigma}^n(\widetilde{X}_s^{1})- \widehat{\sigma}^n(\widetilde{X}_s^{2}) \|^2}{\|\widetilde{Z}_s\|^2} 1_{\{ \|\widetilde{Z}_s\| > \delta \}}ds \right].
\end{eqnarray*}
Using triangle inequality and Jensen's inequality, we get
\begin{eqnarray*}
&& \widetilde{\E}_y \left[ \int_0^{t \wedge \widetilde{D}_n}  \frac{\| \widehat{\sigma}^n(\widetilde{X}_s^{1})- \widehat{\sigma}^n(\widetilde{X}_s^{2}) \|^2}{\|\widetilde{Z}_s\|^2} 1_{\{ \|\widetilde{Z}_s\| > \delta \}}ds \right]  \\
&& \quad \leq 3  \Bigg( \limsup_{m \rightarrow \infty}\widetilde{\E}_y \left[ \int_0^{t \wedge \widetilde{D}_n}  \frac{\| \widehat{\sigma}^n(\widetilde{X}_s^{1})- \widehat{\sigma}^{n,m}(\widetilde{X}_s^{1}) \|^2}{\|\widetilde{Z}_s\|^2} 1_{\{ \|\widetilde{Z}_s\| > \delta \}}ds \right] \\
&& \qquad \qquad + \sup_{m \in \N} \widetilde{\E}_y \left[ \int_0^{t \wedge \widetilde{D}_n}  \frac{\| \widehat{\sigma}^{n,m}(\widetilde{X}_s^{1})- \widehat{\sigma}^{n,m}(\widetilde{X}_s^{2}) \|^2}{\|\widetilde{Z}_s\|^2} 1_{\{ \|\widetilde{Z}_s\| > \delta \}}ds \right] \\
&& \qquad \qquad \qquad+ \limsup_{m \rightarrow \infty} \widetilde{\E}_y \left[ \int_0^{t \wedge \widetilde{D}_n}  \frac{\| \widehat{\sigma}^{n,m}(\widetilde{X}_s^{2})- \widehat{\sigma}^{n}(\widetilde{X}_s^{2}) \|^2}{\|\widetilde{Z}_s\|^2} 1_{\{ \|\widetilde{Z}_s\| > \delta \}}ds \right]  \Bigg).
\end{eqnarray*}
Using Corollary \ref{kryloves},
\begin{eqnarray*}
&& \limsup_{m \rightarrow \infty} \widetilde{\E}_y \left[ \int_0^{t \wedge \widetilde{D}_n}  \frac{\| \widehat{\sigma}^n(\widetilde{X}_s^{1})- \widehat{\sigma}^{n,m}(\widetilde{X}_s^{1}) \|^2}{\|\widetilde{Z}_s\|^2} 1_{\{ \|\widetilde{Z}_s\| > \delta \}}ds \right]  \\
&& \quad \leq \limsup_{m \rightarrow \infty} \frac{1}{\delta^2} \widetilde{\E}_y \left[ \int_0^t  \| \widehat{\sigma}^n(\widetilde{X}_s^{1})- \widehat{\sigma}^{n,m}(\widetilde{X}_s^{1}) \|^2ds \right] \qquad \text{}\\
&& \quad \leq\limsup_{m \rightarrow \infty}   \frac{C_{y,B_{n+1}, \widetilde{q}} }{\delta^2} \sum_{i,j=1}^d \| (\widehat{\sigma}^n_{ij}- \widehat{\sigma}^{n,m}_{ij}) \|_{L^{2\widetilde{q}}(\R^d)}^2 = 0,
\end{eqnarray*} 
where $C_{y,B_{n+1}, \widetilde{q}}>0$ is a constant as in Corollary \ref{kryloves}. In the same way as above, we also get
$$
\limsup_{m \rightarrow \infty} \widetilde{\E}_y \left[ \int_0^{t \wedge \widetilde{D}_n}  \frac{\| \widehat{\sigma}^{n,m}(\widetilde{X}_s^{2})- \widehat{\sigma}^{n}(\widetilde{X}_s^{2}) \|^2}{\|\widetilde{Z}_s\|^2} 1_{\{ \|\widetilde{Z}_s\| > \delta \}}ds \right] =0.
$$
By Proposition \ref{maxth} and Corollary \ref{kryloves},
\begin{eqnarray}
&& \sup_{m \in \N} \widetilde{\E}_y \left[ \int_0^{t \wedge \widetilde{D}_n}  \frac{\| \widehat{\sigma}^{n,m}(\widetilde{X}_s^{1})- \widehat{\sigma}^{n,m}(\widetilde{X}_s^{2}) \|^2}{\|\widetilde{Z}_s\|^2} 1_{\{ \|\widetilde{Z}_s\| > \delta \}}ds \right]  \nonumber \\
&&\leq  \sup_{m \in \N} \sum_{i, j=1}^d \widetilde{c}_d\, \widetilde{\E}_y \left[  \int_0^{t}  \Big( 1_{B_{n}}(\widetilde{X}^1_s) \mathcal{M}\|\nabla \widehat{\sigma}_{ij}^{n,m}\| (\widetilde{X}_s^{1})  +  1_{B_{n}}(\widetilde{X}^2_s)
\mathcal{M}\|\nabla \widehat{\sigma}_{ij}^{n,m}\| (\widetilde{X}_s^{2})\Big)^2  ds \right]  \label{maxest}\\
&&\leq \sup_{m \in \N} \sum_{i,j=1}^d 2 e^{t} \,\widetilde{c}_d\, C_{y,B_{n+1}, \widetilde{q}} \left( \Big \| 1_{B_{n}}\mathcal{M}\| \nabla \widehat{\sigma}_{ij}^{n,m} \|  \Big \|^2_{L^{2\widetilde{q}}(\R^d)} +\Big \| 1_{B_{n}} \mathcal{M}\| \nabla \widehat{\sigma}_{ij}^{n,m} \|  \Big \|^2_{L^{2\widetilde{q}}(\R^d)} \right)  \nonumber \\
&&\leq 4 e^{t}c^2_{d, 2\widetilde{q}}\, \widetilde{c}_d\, C_{y,B_{n+1}, \widetilde{q}} \sup_{m \in \N} \sum_{i,j=1}^d \| \nabla \widehat{\sigma}_{ij}^{n,m}\|^2_{L^{2\widetilde{q}}(\R^d)} \leq 4 e^{t} c^2_{d, 2\widetilde{q}}\, \widetilde{c}_d\, C_{y,B_{n+1}, \widetilde{q}}   \sum_{i,j=1}^d \| \nabla \widehat{\sigma}_{ij}^{n}\|^2_{L^{2\widetilde{q}}(\R^d)},  \nonumber
\end{eqnarray}
where $c_{d, 2\widetilde{q}},\, \widetilde{c}_d>0$ are constants as in Proposition \ref{maxth}. Therefore, $\widetilde{\E}_y \left[ |M_{t\wedge D_n}|^2 \right] < \infty$. Analogously to the above, 
$$
\widetilde{\E}_y \left[ |A^{(1)}_{t\wedge \widetilde{D}_n}| \right]  \leq   4   e^{t} c^2_{d, 2\widetilde{q}}\, \widetilde{c}_d\, C_{y,B_{n+1}, \widetilde{q}}   \sum_{i,j=1}^d \| \nabla \widehat{\sigma}_{ij}^{n}\|^2_{L^{2\widetilde{q}}(\R^d)}< \infty.
$$
Using Fatou's Lemma,
\begin{eqnarray*}
 \frac12 \widetilde{\E}_y \left[ | A^{(2)}_{t \wedge \widetilde{D}_n} |  \right] &\leq& \widetilde{\E}_y \left[ \int_0^{t \wedge \widetilde{D}_n}\frac{\|\mathbf{G}^n(\widetilde{X}_s^{1})  -\mathbf{G}^n(\widetilde{X}_s^{2}) \| }{\|\widetilde{Z}_s \|}   ds \right]\\
&\leq& \liminf_{\delta \rightarrow 0+}  \widetilde{\E}_y \left[\int_0^{t \wedge \widetilde{D}_n} \frac{\|\mathbf{G}^n(\widetilde{X}_s^{1})  -\mathbf{G}^n(X_s^{2}) \| }{\|\widetilde{Z}_s \|}    1_{\{ \|\widetilde{Z}_s\| > \delta \}} ds \right],
\end{eqnarray*}
and again as above, we have
\begin{eqnarray*}
&& \qquad \widetilde{\E}_y \left[ \int_0^{t \wedge D_n}  \frac{\| \mathbf{G}^n(\widetilde{X}_s^{1})- \mathbf{G}^n(\widetilde{X}_s^{2}) \|}{\|\widetilde{Z}_s\|} 1_{\{ \|\widetilde{Z}_s\| > \delta \}}ds \right]  \\
&& \qquad \qquad \leq  \Bigg( \limsup_{m \rightarrow \infty}\widetilde{\E}_y \left[ \int_0^{t \wedge \widetilde{D}_n}  \frac{\| \mathbf{G}^n(\widetilde{X}_s^{1})- \mathbf{G}^{n,m}(\widetilde{X}_s^{1}) \|}{\|\widetilde{Z}_s\|} 1_{\{ \|\widetilde{Z}_s\| > \delta \}}ds \right] \\
&& \qquad \qquad \qquad \quad+ \sup_{m \in \N} \widetilde{\E}_y \left[ \int_0^{t \wedge \widetilde{D}_n}  \frac{\| \mathbf{G}^{n,m}(\widetilde{X}_s^{1})- \mathbf{G}^{n,m}(\widetilde{X}_s^{2}) \|}{\|\widetilde{Z}_s\|} 1_{\{ \|\widetilde{Z}_s\| > \delta \}}ds \right]  \\
&& \qquad \qquad \qquad \qquad + \limsup_{m \rightarrow \infty} \widetilde{\E}_y \left[ \int_0^{t \wedge \widetilde{D}_n}  \frac{\| \mathbf{G}^{n,m}(\widetilde{X}_s^{2})- \mathbf{G}^{n}(\widetilde{X}_s^{2}) \|}{\|\widetilde{Z}_s\|} 1_{\{ \|\widetilde{Z}_s\| > \delta \}}ds \right]  \Bigg). \quad \text{}
\end{eqnarray*}
Using  Corollary \ref{kryloves},
\begin{eqnarray*}
&&\limsup_{m \rightarrow \infty} \widetilde{\E}_y \left[ \int_0^{t \wedge \widetilde{D}_n}  \frac{\| \mathbf{G}^n(\widetilde{X}_s^{1})- \mathbf{G}^{n,m}(\widetilde{X}_s^{1}) \|}{\|\widetilde{Z}_s\|} 1_{\{ \|\widetilde{Z}_s\| > \delta \}}ds \right] \\
&&\leq \limsup_{m \rightarrow \infty} \frac{1}{\delta^2} \widetilde{\E}_y \left[ \int_0^t \| \mathbf{G}^n(\widetilde{X}_s^{1})- \mathbf{G}^{n,m}(\widetilde{X}_s^{1}) \| ds \right] \leq\limsup_{m \rightarrow \infty}   \frac{e^{t} C_{y,B_{n+1}, \widetilde{q}}}{\delta^2} \| \mathbf{G}^n- \mathbf{G}^{n,m} \|_{L^{\widetilde{q}}(\R^d)} = 0,
\end{eqnarray*}
and as before
$$
\limsup_{m \rightarrow \infty} \widetilde{\E}_y \left[ \int_0^{t \wedge \widetilde{D}_n}  \frac{\| \mathbf{G}^n(\widetilde{X}_s^{2})- \mathbf{G}^{n,m}(\widetilde{X}_s^{2}) \|}{\|\widetilde{Z}_s\|} 1_{\{ \|\widetilde{Z}_s\| > \delta \}}ds \right]=0.
$$
Using an estimate as in \eqref{maxest}, Proposition \ref{maxth} and Corollary \ref{kryloves}, we obtain
\begin{eqnarray*}
&&\sup_{m \in \N} \widetilde{\E}_y \left[ \int_0^{t \wedge D_n}  \frac{\| \mathbf{G}^{n,m}(\widetilde{X}_s^{1})- \mathbf{G}^{n,m}(\widetilde{X}_s^{2}) \|}{\|\widetilde{Z}_s\|} 1_{\{ \|\widetilde{Z}_s\| > \delta \}}ds \right]  \leq 2 e^{t} c_{d, 2\widetilde{q}}\, \widetilde{c}_d\, C_{y,B_{n+1}, \widetilde{q}} \sum_{i=1}^d \| \nabla g_i \|_{L^{\widetilde{q}}(\R^d)}.
\end{eqnarray*}
Therefore, $\tilde{\E}_y \left[ | A^{(2)}_{t \wedge D_n} |  \right] <\infty $.
Thus, applying \cite[page 378, (2.3) Proposition]{RYor} to \eqref{undereq}, it follows that
\begin{eqnarray*}
\|\widetilde{Z}_{t \wedge \widetilde{D}_n} \|^2  = \|\widetilde{Z}_0\|^2 \exp\left( M_{t \wedge \widetilde{D}_n} +A^{(1)}_{t \wedge \widetilde{D}_n} +A^{(2)}_{t \wedge \widetilde{D}_n}-\frac{1}{2} \langle M \rangle_{t \wedge \widetilde{D}_n}  \right) =0, \quad \widetilde{\P}_{y}\text{-a.s.}
\end{eqnarray*}
Since $\widetilde{D}_n \rightarrow \infty$\, $\widetilde{\P}_y$-a.s. as $n \rightarrow \infty$, letting $n \rightarrow \infty$ we obtain $\widetilde{Z}_t=0$,\, $\widetilde{\P}_{y}$-a.s, so that pathwise uniqueness is shown. The existence of a strong solution $(Y^y_t)_{t \geq 0}$ to \eqref{sdes} follows from the Yamada-Watanabe Theorem (\cite[Corollary 1]{YW71}).
\vskip 0.1cm \hfill$\Box$\\
\text{}\\
{\rm Proof of Corollary \ref{cor:1.2}.}\;\;
Let  $A=\sigma=id$, $\sqrt{\frac{1}{\psi}}\,(x)=\|x\|^{\alpha/2}$, $x \in \R^d$ and $E=\R^d \setminus \{ 0\}$. Then
{\bf (H1)}--{\bf (H4)} are satisfied and the SDE \eqref{sdes} becomes the SDE \eqref{sdesspe}. Choose $q \in (2d+2, \frac{d}{\alpha})$ and $\widetilde{q} \in (\frac{d}{2}, \frac{d}{2-\alpha} \wedge \frac{d}{2}+\varepsilon)$. Then $\psi \in L^{q}_{loc}(\R^d)$
$\sqrt{\frac{1}{\psi}} \in H^{1,2\widetilde{q}}_{loc}(\R^d)$ and $g_i \in H^{1, \widetilde{q}}_{loc}(\R^d)$ for all $1\leq i \leq d$. Therefore, the assertion follows from Theorem \ref{theo:1.1}. \vskip 0.1cm \hfill$\Box$
\begin{rem}
\; In Theorems \ref{theo:1.1}, \ref{theo:2.1} and Corollary \ref{kryloves}, the condition {\bf (H2)} can be replaced by any ohter condition that implies the non-explosion of $\M$ defined in Section \ref{frame} (cf. \cite{LT19de}).
\end{rem}

\begin{exam}
\;Let $d \geq 3$ and for some $\varepsilon>0$ let $\mathbf{G}=(g_1, \ldots, g_d) \in L_{loc}^{\infty}(\R^d, \R^d)$ with $g_i  \in H^{1,\frac{d}{2}+\varepsilon}_{loc}(\R^d)$ for all $1 \leq i,j \leq d$. Consider the following It\^{o}-SDE
\begin{equation} \label{exdesde1}
X_t = y+\int_0^t \sqrt{\frac{1}{\psi}}\,(X_s) \cdot id \, dW_s + \int_0^t \mathbf{G}(X_s) ds, \qquad 0 \leq t<\infty, \;
\end{equation}
where $y$ and $\sqrt{\frac{1}{\psi}}$ are specified below in three different cases. 
\begin{itemize}
\item[(i)]
For $\alpha \in [0, \frac{d}{2d+2})$ and $\gamma \in (0, \infty)$, let 
$$
\sqrt{\frac{1}{\psi}}(x):=\|x\|^{\alpha/2}1_{\{x\not=0\}}(x)+\gamma 1_{\{x=0\}}(x), \quad x\in \R^d.
$$
Then $\left \{\sqrt{\frac{1}{\psi}}=0\right\}=\emptyset$. Thus by Theorem \ref{theo:1.1} (cf. proof of Corollary \ref{cor:1.2}), for $y \in \R^d \setminus \{0\}$ pathwise uniqueness holds for \eqref{exdesde1} in the usual sense, i.e. if 
$$
(\widetilde{\Omega}, \widetilde{\F}, \widetilde{\P}_{y}, (\widetilde{\F}_t)_{t \ge 0}, (\widetilde{X}^k_t)_{t \ge 0}, (\widetilde{W}_t)_{t \ge 0}), \quad k \in \{1,2 \}
$$
are weak solutions to \eqref{exdesde1}, then
\begin{equation*}
\widetilde{\P}_{y}(\widetilde{X}_t^1 =\widetilde{X}_t^2, \; t \geq 0 )=1.
\end{equation*}
Moreover, by Theorem \ref{theo:1.1} and Corollary \ref{kryloves} for $y \in \R^d \setminus \{0\}$ and a $d$-dimensional Brownian motion $(\widetilde{W}_t)_{t \geq 0}$ on a probability space $(\widetilde{\Omega}, \widetilde{\mathcal{F}}, \widetilde{\P})$
there exists a strong solution $(Y^{y}_t)_{t \geq 0}$ to \eqref{exdesde1} satisfying
\begin{equation*}
\int_0^{\infty} 1_{\{0\}}(Y^y_s) ds =0, \quad \text{ $\widetilde{\P}$-a.s.}
\end{equation*}
\item[(ii)]
Let $\alpha \in [0, \frac{d}{2d+2})$. For $i \in \N \cup \{0\}$, let
$$
\phi_i(x) = \|x -2i \mathbf{e}_1\|^{\frac{\alpha}{2}} 1_{B_{1}(2i\mathbf{e}_1)}(x), \quad x \in \R^d,
$$
where $B_1(2i\mathbf{e}_1)$ denotes an open ball with center $2i\mathbf{e}_1$ and radius $1$. Define
$$
\sqrt{\frac{1}{\psi}}(x):= 1_{\R^d \setminus \bigcup_{i=0}^{\infty} B_{1}(2i\mathbf{e}_1) }+ \sum_{i=0}^{\infty} \phi_i(x), \qquad x \in \R^d.
$$
Note that $\left \{\sqrt{\frac{1}{\psi}}=0\right\}= \{ 2i \mathbf{e}_1: i \in \N \cup \{0\} \}$. Thus by Theorem \ref{theo:1.1} (cf. proof of Corollary \ref{cor:1.2}), for $y \in \R^d \setminus \{ 2i \mathbf{e}_1: i \in \N \cup \{0\} \}$ pathwise uniqueness holds for \eqref{exdesde1} in the following sense, i.e. if
 $$
 (\widetilde{\Omega}, \widetilde{\F}, \widetilde{\P}_{y}, (\widetilde{\F}_t)_{t \ge 0}, (\widetilde{X}^k_t)_{t \ge 0}, (\widetilde{W}_t)_{t \ge 0}),\;\; k \in \{1,2 \}
$$
 are weak solutions to \eqref{sdes} satisfying both
\begin{equation*} 
\int_0^{\infty} 1_{\bigcup_{i=0}^{\infty}\{ 2i \mathbf{e}_1 \}}(\widetilde{X}^k_s) ds =0, \quad \text{ $\widetilde{\P}_{y}$-a.s},
\end{equation*}
then
\begin{equation*}
\widetilde{\P}_{y}(\widetilde{X}_t^1 =\widetilde{X}_t^2, \;\; t \geq 0 )=1.
\end{equation*}
Moreover, by Theorem \ref{theo:1.1} and Corollary \ref{kryloves} for $y \in \R^d \setminus \{ 2i \mathbf{e}_1: i \in \N \cup \{0\} \}$ and a $d$-dimensional Brownian motion $(\widetilde{W}_t)_{t \geq 0}$ on a probability space $(\widetilde{\Omega}, \widetilde{\mathcal{F}}, \widetilde{\P})$
there exists a strong solution $(Y^{y}_t)_{t \geq 0}$ to \eqref{exdesde1} satisfying
\begin{equation} \label{zerotimeat}
\int_0^{\infty} 1_ {\{ 2i \mathbf{e}_1: i \in \N \cup \{0\} \}}(Y^y_s) ds =0, \quad \text{ $\widetilde{\P}$-a.s}.
\end{equation}
\item[(iii)]
Let $(\gamma_i)_{i \geq 0}$ be a sequence with $\gamma_i \in (0, \infty)$ for all $i \in \N \cup \{0\}$. For $\alpha \in [0, \frac{d}{2d+2})$ and $i \in \N \cup \{0\}$, let $\phi_i$ be defined as in (ii).  Define
$$
\sqrt{\frac{1}{\psi}}(x):= 1_{\R^d \setminus \bigcup_{i=0}^{\infty} B_{1}(2i\mathbf{e}_1) }+ \sum_{i=0}^{\infty} \gamma_i 1_{\{ 2i \mathbf{e}_1 \}}+ 1_{\R^d \setminus \bigcup_{i=0}^{\infty}\{ 2i \mathbf{e}_1 \}}\left(\sum_{i=0}^{\infty}  \phi_i(x)\right), \;\quad x \in \R^d.
$$
Then $\left \{\sqrt{\frac{1}{\psi}}=0\right\}=\emptyset$. Therefore, by Theorem \ref{theo:1.1} (cf. proof of Corollary \ref{cor:1.2}), for  $y \in \R^d \setminus \{ 2i \mathbf{e}_1: i \in \N \cup \{0\} \}$ pathwise uniqueness holds for \eqref{exdesde1} in the usual sense and by Corollary \ref{kryloves} there exists a strong solution $(Y^y_t)_{t \geq 0}$ to \eqref{exdesde1} satisfying \eqref{zerotimeat}. 
\end{itemize}
\end{exam}
\text{}\\
{\rm Acknowledgement:} \; The author would like to thank Professor Gerald Trutnau for helpful discussions and suggestions. Additionally, the author is grateful to the anonymous referee for reviewing the manuscript carefully and providing useful comments to improve the paper. \\

\text{}\\
\centerline{}
Haesung Lee\\
Department of Mathematics and Computer Science, \\
Korea Science Academy of KAIST, \\
105-47 Baegyanggwanmun-ro, Busanjin-gu,\\
Busan 47162, Republic of Korea\\
E-mail: fthslt14@gmail.com

\begin{thebibliography}{XXX}

\bibitem{AP07}
J. M. Aldaz, L. J. P\'{e}rez,  Functions of bounded variation, the derivative of the one dimensional maximal function, and applications to inequalities. Trans. Amer. Math. Soc. 359 (2007), no. 5, 2443–2461. 


\bibitem{BC05}
R. Bass,  Z-Q. Chen, {\it One-dimensional stochastic differential equations with singular and degenerate coefficients}, Sankhyā 67 (2005), no. 1, 19–45.


\bibitem{ES85}
H. J. Engelbert, W. Schmidt, 
{\it On solutions of one-dimensional stochastic differential equations without drift}, Z. Wahrsch. Verw. Gebiete 68 (1985), no. 3, 287–314.


\bibitem{ES85b}
H. J. Engelbert, W. Schmidt, 
{\it On one-dimensional stochastic differential equations with generalized drift}, Stochastic differential systems (Marseille-Luminy, 1984), 143–155, Lect. Notes Control Inf. Sci., 69, Springer, Berlin, 1985.


\bibitem{ES91}
H. J. Engelbert, W. Schmidt,  
{\it Strong Markov continuous local martingales and solutions of one-dimensional stochastic differential equations},
III. Math. Nachr. 151 (1991), 149–197.


\bibitem{LT20}  H. Lee, W. Stannat, G. Trutnau, {\it  Analytic theory of It\^{o}-stochastic differential equations with non-smooth coefficients}, to appear in the SpringerBriefs in Probability and Mathematical Statistics, arXiv:2012.14410v4.


\bibitem{LT19de} H. Lee, G. Trutnau, {\it Well-posedness for a class of degenerate It\^{o} Stochastic Differential Equations with fully discontinuous coefficients}, Symmetry 2020, 12(4), 570, arXiv:1909.09430v2.

 
 \bibitem{MS73}
S. Manabe, T. Shiga, {\it On one-dimensional stochastic differential equations with non-sticky boundary conditions},  J. Math. Kyoto Univ. 13 (1973), 595–603. 

 
\bibitem{Kry} N. V. Krylov, {\it Controlled Diffusion Processes}, Applications of Mathematics, 14. Springer-Verlag, New York-Berlin, 1980.

 
\bibitem{RYor} 
D. Revuz, M. Yor, {\it Continuous martingales and Brownian motion}, Springer Verlag, (2005).


\bibitem{S93}
E. M. Stein, {\it  Singular integrals and differentiability properties of functions}, Princeton Mathematical Series, No. 30 Princeton University Press, Princeton, N.J. 1970. 


\bibitem{YW71}
T. Yamada, S. Watanabe, {\it On the uniqueness of solutions of stochastic differential equations}, J. Math. Kyoto Univ. 11, 155–167, (1971).


\bibitem{X11}
X. Zhang, {\it Stochastic homeomorphism flows of SDEs with singular drifts and Sobolev diffusion coefficients}, Electron. J. Probab. 16 (2011), no. 38, 1096–1116.


\bibitem{X13}
X. Zhang, {\it Well-posedness and large deviation for degenerate SDEs with Sobolev coefficients},  Rev. Mat. Iberoam. 29 (2013), no. 1, 25–52.





\end{thebibliography}
\end{document}